\documentclass[12pt,draftcls,onecolumn]{IEEEtran}

\usepackage{hyperref}

\usepackage{relsize}
\usepackage{amsmath}
\usepackage{amsfonts}
\usepackage{framed}
\usepackage{pgf}
\usepackage{relsize}
\usepackage[center]{caption}
 
\usepackage{enumitem}
\usepackage{graphicx}
\newtheorem{assumption}{Assumption}

\newcommand{\be}{\begin{equation}}
\newcommand{\ee}{\end{equation}}

\newcommand{\argmin}{\mathop{\rm argmin}}
\newcommand{\argmax}{\mathop{\rm argmax}}

\usepackage{setspace}

\IEEEoverridecommandlockouts                              
\newtheorem{theorem}{Theorem}[section]
\newtheorem{lemma}[theorem]{Lemma}

\newcommand{\norm}[1]{\left|\left|#1\right|\right|}

\newlength{\smpagewidth}
\newlength{\smpageheight}

\newcommand{\setleftmargin}[1]{
    \addtolength{\textwidth}{\oddsidemargin}
    \addtolength{\textwidth}{1in}
    \addtolength{\textwidth}{-#1}
    \setlength{\oddsidemargin}{-1in}
    \addtolength{\oddsidemargin}{#1}
    \setlength{\evensidemargin}{\oddsidemargin}
}
\newcommand{\setrightmargin}[1]{
    \setlength{\textwidth}{\smpagewidth}
    \addtolength{\textwidth}{-\oddsidemargin}
    \addtolength{\textwidth}{-1in}
    \addtolength{\textwidth}{-#1}
}
\newcommand{\settopmargin}[1]{
    \addtolength{\textheight}{\topmargin}
    \addtolength{\textheight}{1in}
    \addtolength{\textheight}{\headheight}
    \addtolength{\textheight}{\headsep}
    \addtolength{\textheight}{-#1}
    \setlength{\topmargin}{-1in}
    \addtolength{\topmargin}{-\headheight}
    \addtolength{\topmargin}{-\headsep}
    \addtolength{\topmargin}{#1}
}
\newcommand{\setbottommargin}[1]{
    \setlength{\textheight}{\smpageheight}
    \addtolength{\textheight}{-\topmargin}
    \addtolength{\textheight}{-1in}
    \addtolength{\textheight}{-\footskip}
    \addtolength{\textheight}{-#1}
}

\setleftmargin{.75in}
\setrightmargin{.75in}
\settopmargin{.75in}
\setbottommargin{.77in}

\makeatother

\title{On the $O(1/k)$ Convergence of Asynchronous Distributed Alternating Direction Method of Multipliers$^*$\thanks{$^*$This work was supported by National Science Foundation under Career grant
DMI-0545910, AFOSR MURI FA9550-09-1-0538, and ONR Basic Research Challenge No. N000141210997.}}

\author{Ermin Wei$^\dag$\thanks{$^\dag$Department of Electrical Engineering and Computer Science, Massachusetts Institute of Technology} and
Asuman Ozdaglar$^\dag$}
\date{\today}
\begin{document}\maketitle

\begin{abstract}
We consider a network of agents that are cooperatively solving a global optimization problem, where the objective function is the sum of privately known local objective functions of the agents and
the decision variables are coupled via linear constraints. Recent literature focused on special cases of this formulation and  studied their distributed solution through either subgradient based
methods with $O(1/\sqrt{k})$ rate of convergence (where $k$ is the iteration number) or Alternating Direction Method of Multipliers (ADMM) based methods, which require a synchronous implementation
and a globally known order on the agents. In this paper, we present a novel asynchronous ADMM based distributed method for the general formulation and show that it converges at the rate
$O\left(1/k\right)$.
\end{abstract}

\section{Introduction}\label{sec:intro}

We consider the following optimization problem with a separable objective function and linear constraints:
\begin{align}\label{Multisplit-formulation}
\min_{x_i\in X_i,z\in Z}\quad &\sum_{i=1}^N f_i(x_i)\\ \nonumber
s.t. \quad & Dx + Hz= 0. 
\end{align}	
Here each $f_i:\mathbb{R}^n\to \mathbb{R}$ is a (possibly nonsmooth) convex function, $X_i$ and $Z$ are closed convex subsets of $\mathbb{R}^{n}$ and $\mathbb{R}^{W}$, and $D$ and $H$ are matrices
of dimensions $W \times nN$ and $W \times W$. The decision variable $x$ is given by the partition $x = [x_1', \ldots, x_{N}']' \in \mathbb{R}^{nN}$, where the $x_i\in \mathbb{R}^n$ are components
(subvectors) of $x$. We denote by set $X$ the product of sets $X_i$, hence the constraint on $x$ can be written compactly as $x\in X$.

Our focus on this formulation is motivated by {\it distributed multi-agent optimization problems}, which attracted much recent attention in the optimization, control and signal processing
communities. Such problems involve resource allocation, information processing, and learning among a set $\{1,\ldots,N\}$ of distributed agents connected through a network $G=(V,E)$, where $E$
denotes the set of $M$ undirected edges between the agents. In such applications, each agent has access to a privately known local objective (or cost) function, which represents the negative utility
or the loss agent $i$ incurs at the decision variable $x$. The goal is to collectively solve a global optimization problem\footnote{The usefulness of formulation (\ref{distopt-problem}) can be
illustrated by, among other things, machine learning problems described as follows:
 \begin{align*}
\min_x \sum_{i=1}^{N-1} l\left([W_ix-b_i]\right)+\pi \norm{x}_1,
\end{align*}
where $W_i$ corresponds to the input sample data (and functions thereof), $b_i$ represents the measured outputs, $W_i x-b_i$ indicates the prediction error and $l$ is the loss function on the
prediction error. Scalar $\pi$ is nonnegative and it indicates the penalty parameter on complexity of the model. The widely used Least Absolute Deviation (LAD) formulation,  the Least-Absolute
Shrinkage and Selection Operator (Lasso) formulation and $l_1$ regularized formulations can all be represented by the above formulation by varying loss function $l$ and penalty parameter $\pi$ (see
\cite{BishopML} for more details). The above formulation is a special case of the distributed multi-agent optimization problem (\ref{distopt-problem}), where $f_i(x) = l\left(W_ix-b_i\right)$ for
$i=1,\ldots, N-1$ and $f_{N} = \pi_2\norm{x}_1.$ In applications where the data pairs $\big(W_i, b_i\big)$ are collected and maintained by different sensors over a network, the functions $f_i$ are
local to each agent and the need for a distributed algorithm arises naturally. }
 \begin{align}\label{distopt-problem}
\min &\sum_{i=1}^N f_i(x)\\ \nonumber
s.t. \quad & x\in X. 
\end{align}	
This problem can be reformulated in the general formulation of (\ref{Multisplit-formulation}) by introducing a local copy $x_i$ of the decision variable for each node $i$  and imposing the
constraint $x_i=x_j$ for all agents $i$ and $j$ with edge $(i,j)\in E$. Under the assumption that the underlying network is connected, this condition ensures that each of the local copies are equal
to each other. Using the edge-node incidence matrix of network $G$, denoted by $A \in \mathbb{R}^{Mn  \times Nn }$, the reformulated problem can be written compactly as\footnote{ The edge-node
incidence matrix of network $G$ is defined as follows:  Each $n$-row block of matrix $A$ corresponds to an edge in the graph and each $n$-column block represent a node. The $n$ rows corresponding to
the edge $e=(i,j)$ has $I(n\times n)$ in the $i^{th}$ $n$-column block, $-I(n\times n)$ in the $j^{th}$ $n-$column block and $0$ in the other columns, where $I(n\times n)$ is the identity matrix of
dimension $n$. }
\begin{align}\label{distFormulationSync}
\min_{x_i\in X}\quad &\sum_{i=1}^N f_i(x_i)\\ \nonumber
s.t. \quad & A x = 0, \nonumber
\end{align}
where $x$ is the vector $[x_1, x_2, \ldots, x_N]'$. We will refer to this formulation as the {\it edge-based reformulation} of the multi-agent optimization problem. Note that this formulation is a
special case of problem  (\ref{Multisplit-formulation})  with $D=A$, $H=0$ and $X_i=X$ for all $i$. Since these problems often lack a centralized processing unit, it is imperative that iterative
solutions of problem (\ref{distopt-problem}) involve decentralized computations, meaning that each node (processor) performs calculations independently and on the basis of local information
available to it and then communicates this information to its neighbors according to the underlying network structure.

Though there have been many important advances in the design of decentralized optimization algorithms for multi-agent optimization problems, several challenges still remain. First, many of these
algorithms are based on first-order subgradient methods, which have slow convergence rates (given by $O(1/\sqrt{k})$ where $k$ is the iteration number), making them impractical in many large scale
applications. Second, with the exception of a few recent contributions, existing algorithms are synchronous, meaning that computations are simultaneously performed according to some global clock,
but this often goes against the highly decentralized nature of the problem, which precludes such global information being available to all nodes.

In this paper, we focus on the more general formulation (\ref{Multisplit-formulation}) and propose an asynchronous decentralized algorithm based on the classical Alternating Direction Method of
Multipliers (ADMM) (see \cite{ADMMBoyd}, \cite{Eckstein2012} for comprehensive tutorials). We adopt the following asynchronous implementation for our algorithm: at each iteration $k$, a random
subset $\Psi^k$ of the constraints are selected, which in turn selects the components of $x$ that appear in these constraints. We refer to the selected constraints as {\it active constraints} and
selected components as the {\it active components (or agents)}. We design an ADMM-type primal-dual algorithm which at each iteration  updates the primal variables using partial information about the
problem data, in particular using cost functions corresponding to active components and active constraints, and updates the dual variables corresponding to active constraints. In the context of the
edge-based reformulated multi-agent optimization problem (\ref{distFormulationSync}), this corresponds to a fully decentralized and asynchronous implementation in which a subset of the edges are
randomly activated (for example according to local clocks associated with those edges) and the agents incident to those edges perform computations on the basis of their local objective functions
followed by communication of updated values with neighbors.

Under the assumption that each constraint has a positive probability of being selected and the constraints have a decoupled structure (which is satisfied by reformulations of the distributed
multi-agent optimization problem), our first result shows that the (primal) asynchronous iterates generated by this algorithm converge almost surely to an optimal solution. Our proof relies on
relating the asynchronous iterates to {\it full-information} iterates that would be generated by the algorithm that use full information about the cost functions and constraints at each iteration.
In particular, we introduce a weighted norm where the weights are given by the inverse of the probabilities with which the constraints are activated and constructs a Lyapunov function for the
asynchronous iterates using this weighted norm. Our second result establishes a  {\it performance guarantee of $O(1/k)$} for this algorithm under a compactness assumption on the constraint sets $X$
and $Z$, which to our knowledge is faster than the guarantees available in the literature for this problem. More specifically, we show that the expected value of the difference of the objective
function value and the optimal value as well as the expected feasibility violation converges to 0 at rate $O(1/k)$.

Our paper is related to a large recent literature on distributed optimization methods for solving the multi-agent optimization problem. Most closely related is a recent stream which proposed
distributed synchronous ADMM algorithms for solving problem (\ref{distopt-problem}) (or specialized versions of it) (see \cite{MotaColoring}, \cite{Mota2012}, \cite{JXMAug}, \cite{giannakisAdHoc},
\cite{GiannakisChannelDecoding}). These papers have demonstrated the excellent computational performance of ADMM algorithms in the context of several signal processing applications. A closely
related work in this stream is our recent paper \cite{WeiCDC}, where we considered problem (\ref{distopt-problem}) under the general assumption that the $f_i$ are convex. In \cite{WeiCDC}, we
presented an ADMM based algorithm which operates by updating the decision variable $x$ in $N$ steps in a synchronous manner using a deterministic cyclic order and showed that it converges at the
rate $O(1/k)$. This algorithm however requires a synchronous implementation and a globally known order on the set of agents. The algorithm presented here selects a subset of the components of the
decision variable $x$ randomly and updates the variables $(x,z)$ in two steps by first updating the selected components of $x$ and then updating the $z$ variable.

Another strand of this literature uses first-order (sub)gradient methods for solving problem (\ref{distopt-problem}). Much of this work builds on the seminal works \cite{ParallelCompute} and
\cite{TsitsiklisThes}, which proposed gradient methods that can parallelize computations across multiple processors. The more recent paper \cite{NOSubgradient} introduced a first-order primal
subgradient method for solving problem (\ref{distopt-problem}) over deterministically varying networks. This method involves each agent maintaining and updating an estimate of the optimal solution
by linearly combining a subgradient step along its local cost function with averaging of estimates obtained from his neighbors (also known as a single consensus step).\footnote{This work is clearly
also related to the extensive literature on consensus and cooperative control, where the goal is to design local deterministic or random update rules to achieve global coordination (for
deterministic update rules, see \cite{Blondel}, \cite{BCM2009distributed},  \cite{FCFZ}, \cite{ali}, \cite{KarMoura},  \cite{murray}, \cite{Consensus2},  \cite{OCB}, \cite{Tahbaz-SalehiJad08}; for random update rules, see \cite{ScaglioneGossip}, \cite{BoydGossip}, \cite{GeographicGossip}, \cite{Fagnani08}).} Several follow-up papers
considered variants of this method for problems with local and global constraints \cite{JJSubgradient}, \cite{NOPConstrained} randomly varying networks \cite{LORandomNetwork},
\cite{LobelOzdaglarFeijer}, \cite{MateiBaras} and random gradient errors \cite{SNVStochasticGradient}, \cite{NOOTQuantization}.  A different distributed algorithm that relies on Nesterov's dual
averaging algorithm \cite{NesterovDualAvg} for static networks has been proposed and analyzed in \cite{Duchi2012}. Such gradient methods typically have a convergence rate of $O(1/\sqrt{k})$. The
more recent contribution \cite{MouraFastGrad} focuses on a special case of (\ref{distopt-problem}) under smoothness assumptions on the cost functions and availability of global information about some
problem parameters, and provided gradient algorithms (with multiple consensus steps) which converge at the faster rate of $O(1/k^2)$.

With the exception of \cite{asyncGossip} and \cite{IutzelerRand}, all algorithms provided in the literature are synchronous and assume that computations at all nodes are performed simultaneously
according to a global clock. \cite{asyncGossip} provides an asynchronous subgradient method that uses gossip-type activation and communication between pairs of nodes and shows (under a compactness
assumption on the iterates) that the iterates generated by this method converge almost surely to an optimal solution. The recent independent paper \cite{IutzelerRand} provides an asynchronous
randomized ADMM algorithm for solving problem (\ref{distopt-problem}) and establishes convergence of the iterates to an optimal solution by studying the convergence of randomized Gauss-Seidel
iterations on non-expansive operators. Our paper instead proposes an asynchronous ADMM algorithm  for the more general problem (\ref{Multisplit-formulation}) and uses a Lyapunov function argument
for establishing $O(1/k)$ rate of convergence.

Our algorithm and analysis also build on and combines ideas from several important contributions in the study of ADMM algorithms. Earlier work in this area focuses on the case $C=2$, where $C$
refers to the number of sequential primal updates at each iteration, and studies convergence in the context of finding zeros of the sum of two maximal monotone operators (more specifically, the
Douglas-Rachford operator), see \cite{DouglasRachford}, \cite{Eckstein1989},
\cite{LionsSplitting}. The recent contribution \cite{HeYuanRate} considered solving problem  (\ref{Multisplit-formulation})  (with $C=2$) with ADMM and showed that
the objective function values of the iterates converge at the rate $O(1/k)$. Other recent works analyzed the rate of convergence of ADMM and other related algorithms under smoothness conditions on
the objective function (see \cite{DengYing}, \cite{MaFastSplitting},  \cite{Goldfarb2012}).   Another paper \cite{HanYuanNote2012} considered the  case $C\ge 2$ and showed that the resulting ADMM
algorithm, converges under the more restrictive assumption that each $f_i$ is strongly convex. The recent paper \cite{Luo2012} focused on the general case $C\ge 2$ and established a global linear
convergence rate using an error bound condition that estimates the distance from the dual optimal solution set in terms of norm of a proximal residual.

The paper is organized as follows: we start in Section \ref{sec:ADMMSt} by highlighting the main ideas of the standard ADMM algorithm. In Section \ref{sec:stochADMM2SplittingGeneral}, we focus on
the more general formulation (\ref{Multisplit-formulation}), present  the asynchronous ADMM algorithm and apply this algorithm to solve problem (\ref{distopt-problem}) in a distributed way. Section
\ref{sec:2splitconvGeneral} contains our convergence and rate of convergence analysis. Section \ref{sec:con} concludes with closing remarks.

\noindent\textbf{Basic Notation and Notions:}

A vector is viewed as a column vector. For a matrix $A$, we write $[A]_i$ to denote the $i^{th}$ column of matrix $A$, and $[A]^j$ to denote the $j^{th}$ row of matrix $A$. For a vector $x$, $x_i$
denotes the $i^{th}$ component of the vector. For a vector $x$ in $\mathbb{R}^n$ and set $S$ a subset of $\{1,\ldots, n\}$, we denote by $[x]_{S}$ a vector in $\mathbb{R}^n$, which places zeros for
all components of $x$ outside set $S$, i.e.,
\begin{align*}
[[x]_{S}]_i= \left\{\begin{array}{cc}x_i & \mbox{if}\quad i\in S,\\ 0& \mbox{otherwise.}\end{array}\right.
\end{align*}
 We use $x'$ and $A'$ to
denote the transpose of a vector $x$ and a matrix $A$ respectively.
 We use standard Euclidean norm (i.e., 2-norm) unless otherwise noted, i.e., for a vector $x$ in $\mathbb{R}^n$, $\norm{x}=\left(\sum_{i=1}^n x_i^2\right)^{\frac{1}{2}}$.

\section{Preliminaries: Standard ADMM Algorithm}\label{sec:ADMMSt}
The standard ADMM algorithm solves a separable convex optimization problem where the decision vector decomposes into two variables and the objective function is the sum of convex functions over
these variables that are coupled through a linear constraint:\footnote{Interested readers can find more details in \cite{ADMMBoyd} and \cite{Eckstein2012}.}
\begin{align}\label{ADMMFormulation}
\min_{x\in X,z\in Z}\quad & F_s(x)+G_s(z)\\ \nonumber
s.t. \quad & D_s x + H_s z = c,
\end{align}
where $F_s:\mathbb{R}^n\to \mathbb{R}$ and $G_s:\mathbb{R}^n\to \mathbb{R}$ are convex functions, $X$ and $Z$ are nonempty closed convex subsets of $\mathbb{R}^n$ and $\mathbb{R}^m$, and $D_s$ and
$H_s$ are matrices of dimension $w\times n$ and $w \times m$.

We consider the augmented Lagrangian function of problem (\ref{ADMMFormulation}) obtained by adding a quadratic penalty for feasibility violation to the Lagrangian function:
\begin{align}\label{eq:augL}L_\beta(x,z,p) = F_s(x)+G_s(z)-p'(D_s x+H_s z-c)+\frac{\beta}{2}\norm{D_s x+H_s z-c}^2,\end{align} where $p$ in $\mathbb{R}^w$ is the Lagrange multiplier corresponding
to
the constraint $D_s x+H_s z=c$ and $\beta$ is a positive penalty parameter.

The standard ADMM algorithm is an iterative primal-dual algorithm, which can be viewed as an approximate version of the classical augmented Lagrangian method for solving problem
(\ref{ADMMFormulation}). It proceeds by approximately minimizing the augmented Lagrangian function through updating the primal variables $x$ and $z$ sequentially within a single pass block
coordinate descent (in a Gauss-Seidel manner) at the current Lagrange multiplier (or the dual variable) followed by updating the dual variable through a gradient ascent method (see \cite{Luo2012}
and \cite{Eckstein2012}). More specifically, starting from some initial vector $(x^0, z^0, p^0)$,\footnote{We use superscripts to denote the iteration number.} at iteration $k\ge 0$, the variables
are updated as
\begin{align}
x^{k+1} &\in \argmin_{x\in X} L_\beta(x,z^k,p^k), \label{eq:yUpdate}\\ z^{k+1} & \in \argmin_{z\in Z} L_\beta(x^{k+1},z,p ^k), \label{eq:zUpdate}\\ p ^{k+1}& = p ^k -\beta(D_s x^{k+1}+H_s
z^{k+1}-c). \label{eq:muUpdate}
\end{align}
We assume that the minimizers in steps (\ref{eq:yUpdate}) and (\ref{eq:zUpdate}) exist, however they need not be unique. Note that the stepsize used in updating the dual variable is the same as the
penalty parameter $\beta$.

The ADMM algorithm takes advantage of the separable structure of problem (\ref{ADMMFormulation}) and decouples the minimization of functions $F_s$ and $G_s$ since the sequential minimization over
$x$ and $z$ involves (quadratic perturbations) these functions separately. This is particularly useful in applications where the minimization over these component functions admit simple solutions
and can be implemented in a parallel or decentralized manner.

The analysis of the ADMM algorithm adopts the following standard assumption on problem (\ref{ADMMFormulation}).

\begin{assumption}{\it (Existence of a Saddle Point)}\label{assm:saddlePointSt}
The Lagrangian function of problem (\ref{ADMMFormulation}) given by \[ L(x,z,p) = F_s(x)+G_s(z)-p'(D_sx+H_sz),\] has a saddle point, i.e., there exists a solution-multiplier pair $(x^*, z^*, p^*)$
with
\[ L(x^*,z^*,p)\leq L(x^*,z^*,p^*)\leq L(x,z,p^*),\] for all $x$ in $\mathbb{R}^{n}$, $z$ in $\mathbb{R}^{m}$ and $p$ in $\mathbb{R}^{w}$.
\end{assumption}

Note that the existence of a saddle point is equivalent to the existence of a primal dual optimal solution pair. It is well-known that under the given assumptions, the objective function value of
the primal  sequence $\{x^k,z^k\}$ generated by (\ref{eq:yUpdate})-(\ref{eq:zUpdate}) converges to the optimal value of problem (\ref{ADMMFormulation}) and the dual sequence $\{p^k\}$ generated by
(\ref{eq:muUpdate}) converges to a dual optimal solution (see Section 3.2 of  \cite{ADMMBoyd}).

\section{Asynchronous ADMM Algorithm}\label{sec:stochADMM2SplittingGeneral}

Extending the standard ADMM, we present in this section an asynchronous distributed ADMM algorithm. We present the problem formulation and assumptions  in Section \ref{sec:formulation}. In Section
\ref{sec:asyncAlg}, we discuss the asynchronous implementation considered in the rest of this paper that involves updating a subset of components of the decision vector at each time using partial
information
about problem data and without need for a global coordinator.  Section \ref{sec:generalAlg} contains the details of the asynchronous ADMM algorithm. In Section \ref{sec:specialCase}, we apply the
asynchronous ADMM algorithm to solve the distributed multi-agent optimization problem (\ref{distopt-problem}).

\subsection{Problem Formulation and Assumptions}\label{sec:formulation}

We consider the optimization problem given in (\ref{Multisplit-formulation}), which is restated here for convenience:
\begin{align*}
\min_{x_i\in X_i,z\in Z}\quad &\sum_{i=1}^N f_i(x_i)\\ \nonumber
s.t. \quad & Dx + Hz= 0.
\end{align*}	
This problem formulation arises in large-scale multi-agent (or processor) environments where problem data is distributed across $N$ agents, i.e., each agent has access only to the component function
$f_i$ and maintains the decision variable component $x_i$. The constraints usually represent the coupling across components of the decision variable imposed by the underlying connectivity among the
agents. Motivated by such applications, we will refer to each component function $f_i$ as the {\it local objective function} and use the notation $F:\mathbb{R}^{nN}\to \mathbb{R}$ to denote the {\it
global objective function} given by their sum:
\be\label{eq:defF} F(x)=\sum_{i=1}^N f_i(x_i).\ee

Similar to the standard ADMM formulation, we adopt the following assumption.

\begin{assumption}{\it (Existence of a Saddle Point)}\label{assm:saddlePoint2Split}
The Lagrangian function of problem (\ref{Multisplit-formulation}), \be L(x,z,p) = F(x)-p'(Dx+Hz),\label{eq:LagGeneral}\ee has a saddle point, i.e., there exists a solution-multiplier pair $(x^*,
z^*, p^*)$ with
\be L(x^*,z^*,p)\leq L(x^*,z^*,p^*)\leq L(x,z,p^*)\label{ineq:saddlePoint2}\ee for all $x$ in $X$, $z$ in $Z$ and $p$ in $\mathbb{R}^{W}$.
\end{assumption}

Moreover, we assume that the matrices have special structure that enables solving problem (\ref{Multisplit-formulation}) in an asynchronous manner:

\begin{assumption}\label{assm:matrices} ({\it Decoupled Constraints})
Matrix $H$ is diagonal and invertible. Each row of matrix $D$ has exactly one nonzero element and matrix $D$ has no columns of all zeros.\footnote{We assume without loss of generality that each
$x_i$ is involved at least in one of the constraints, otherwise, we could remove it from the problem and optimize it separately. Similarly, the diagonal elements of matrix $H$ are assumed to be
non-zero, otherwise, that component of variable $z$ can be dropped from the optimization problem.}
\end{assumption}

The diagonal structure of matrix $H$ implies that each component of vector $z$ appears in exactly one linear constraint. The conditions that each row of matrix $D$ has only one nonzero element and
matrix $D$ has no column of zeros guarantee the columns of matrix $D$ are linearly independent and hence matrix $D'D$ is positive definite. The condition on matrix $D$ implies that each row of the
constraint $Dx+Hz=0$ involves exactly one $x_i$. We will see in Section \ref{sec:specialCase} that this assumption is satisfied by the distributed multi-agent optimization problem that motivates
this work.

\subsection{Asynchronous Algorithm Implementation}\label{sec:asyncAlg}

In the large scale multi-agent applications descried above, it is essential that the iterative solution of the problem involves computations performed by agents in a decentralized manner (with
access to local information) with as little coordination as possible. This necessitates an asynchronous implementation in which some of the agents become active (randomly) in time and update the
relevant components of the decision variable using partial and local information about problem data while keeping the rest of the components of the decision variable unchanged. This removes the need
for a centralized coordinator or global clock, which is an unrealistic requirement in such decentralized environments.

To describe the asynchronous algorithm implementation we consider in this paper more formally, we first introduce some notation. We call a partition of the set $\{1,\ldots, W\}$ a {\it proper
partition} if it has the property that if $z_i$ and $z_j$ are coupled in the constraint set $Z$, i.e., value of $z_i$ affects the constraint on $z_j$ for any $z$ in set $Z$, then $i$ and $j$ belong
to the same partition, i.e., $\{i,j\}\subset \psi$ for some $\psi$ in the partition. We let $\Pi$ be a proper partition  of the set $\{1,\ldots, W\}$ , which forms a partition of the set of $W$ rows
of the linear constraint  $Dx+Hz=0$. For each $\psi$ in $\Pi$, we define $\Phi(\psi)$ to be the set of indices $i$, where $x_i$ appears in the linear constraints in set $\psi$. Note that
$\Phi(\psi)$ is an element of the power set $2^{\{1,\ldots, N\}}$.

At each iteration of the asynchronous algorithm, two random variables $\Phi^k$ and $\Psi^k$ are realized. While the pair $(\Phi^k, \Psi^k)$ is correlated for each iteration $k$, these variables are
assumed to be independent and identically distributed across iterations. At each iteration $k$, first the random variable $\Psi^k$ is realized. The realized value, denoted by $\psi^k$, is an element
of the proper partition $\Pi$ and selects a subset of the linear constraints $Dx+Hz=0$. The random variable $\Phi^k$ then takes the realized value $\phi^k=\Phi(\psi^k)$. We can view this process as
activating a subset of the coupling constraints and the components that are involved in these constraints. If $l\in \psi^k$,  we say constraint $l $ as well as its associated dual variable $p_l$ is
{\it active} at iteration $k$. Moreover, if $i\in \Phi(\psi^k)$, we say that component $i$ or agent $i$ is {\it active} at iteration $k$. We use the notation $\bar{\phi}^k$ to denote the complement
of set $\phi^k$ in set $\{1,\ldots,N\}$ and similarly $\bar{\psi}^k$ to denote the complement of set $\psi^k$ in set $\{1,\ldots, W\}$.

Our goal is to design an algorithm in which at each iteration $k$, only active components of the decision variable and active dual variables are updated using local cost functions of active agents
and active constraints. To that end, we define $f^k:\mathbb{R}^{nN}\to \mathbb{R}$ as the sum of the local objective functions whose indices are in the subset $\phi^k$:
\[f^k(x) = \sum_{i\in \phi^k}f_i(x_i), \]
We denote by $D_i$ the matrix in $\mathbb{R}^{W\times nN}$ that picks up the columns corresponding to $x_i$ from matrix $D$ and has zeros elsewhere. Similarly, we denote by $H_l$ the diagonal matrix
in  $\mathbb{R}^{W\times W}$ which picks up the element in the $l^{th}$ diagonal position from matrix $H$ and has zeros elsewhere. Using this notation, we define the matrices
\[D_{\phi^k}=\sum_{i\in \phi^k} D_i, \quad \mbox{and}\quad H_{\psi^k} = \sum_{l\in \psi^k} H_l.\]

We impose the following condition on the asynchronous algorithm.
\begin{assumption}\label{assm:io}(\it{Infinitely Often Update})
For all $k$ and all $\psi$ in  the proper partition $\Pi$, \[\mathbb{P}(\Psi^k=\psi)>0.\]
\end{assumption}

This assumption ensures that each element of the partition $\Pi$ is active infinitely often with probability 1. Since matrix $D$ has no columns of all zeros, each of the $x_i$ is involved in some
constraints, and hence $\cup_{\psi\in\Pi}\Phi(\psi) = \{1,\ldots,N\}$. The preceding assumption therefore implies that each agent $i$ belongs to at least one set $\Phi(\psi)$ and therefore is active
infinitely often with probability $1$. From definition of the partition $\Pi$, we have $\cup_{\psi\in\Pi} \psi= \{1,\ldots,W\}$. Thus, each constraint $l$ is  active infinitely often with
probability $1$. 

\subsection{Asynchronous ADMM Algorithm}\label{sec:generalAlg}
We next describe the asynchronous ADMM algorithm for solving problem (\ref{Multisplit-formulation}).

\begin{framed}\noindent \textbf{I. Asynchronous ADMM algorithm:}
\begin{itemize}
\item[A] Initialization: choose some arbitrary $x^0$ in $X$, $z^0$ in $Z$ and $p^0=0$.
\item[B] At iteration $k$, random variables $\Phi^k$ and $\Psi^k$ takes realizations $\phi^k$ and $\psi^k$. Function $f^k$ and matrices $D_{\phi^k}$, $H_{\psi^k}$ are generated accordingly.
\begin{itemize}
	\item[a]  The primal variable $x$ is updated as \be \label{x2}x^{k+1}\in\argmin_{x\in X} f^k(x)-( p^k)' D_{\phi^k} x+\frac{\beta}{2} \norm{D_{\phi^k} x+ H z^k}^2.\ee with
$x_i^{k+1}=x_i^k$, for $i$ in $\bar\phi^k$.
    \item[b]  The primal variable $z$ is updated as  \be\label{z2}z^{k+1} \in \argmin_{z\in Z}  -(p^k)'H_{\psi^k}z+        \frac{\beta}{2}\norm{H_{\psi^k}z+ D_{\phi^k}x^{k+1}}^2.\ee with
        $z_i^{k+1}=z_i^k$, for $i$ in $\bar\psi^k$.
	\item[c]  The dual variable $p$ is updated as 	\be\label{eq:pUpdate2SplitGeneral} p^{k+1} = p^k-\beta[D_{\phi^k}x^{k+1}+H_{\psi^k}z^{k+1}]_{\psi^k}.\ee
\end{itemize}
\end{itemize}
\end{framed}

We assume that the minimizers in updates (\ref{x2}) and (\ref{z2}) exist, but need not be unique.\footnote{Note that the optimization in (\ref{x2}) and
(\ref{z2}) are independent of components of $x$ not in $\phi^k$ and components of $z$ not in $\psi^k$ and thus the restriction of $x_i^{k+1}=x_i^k$, for $i$ not in $\phi^k$ and
$z_i^{k+1}=z_i^k$, for $i$ not in $\psi^k$ still preserves optimality of $x^{k+1}$ and $z^{k+1}$ with respect to the optimization problems  in update (\ref{x2}) and
(\ref{z2}).} The term $\frac{\beta}{2}\norm{D_{\phi^k} x+ H z^k}^2$ in the objective function of the minimization problem in update (\ref{x2}) can be written as
\[\frac{\beta}{2}\norm{D_{\phi^k} x+ H z^k}^2=\frac{\beta}{2}\norm{D_{\phi^k}x}^2+\beta( Hz^k)'D_{\phi^k}x+\frac{\beta}{2}\norm{Hz^k}^2,\]
where the last term is independent of the decision variable $x$ and thus can be dropped from the objective function. Therefore, the primal $x$ update can be written as
\begin{align}\label{eq:xUpdate2SplitGeneral} x^{k+1} \in\argmin_{x \in X} f^k(x)-(p^k-\beta Hz^k)'D_{\phi^k}x+\frac{\beta}{2}\norm{D_{\phi^k}x}^2.\end{align}
Similarly, the term $ \frac{\beta}{2}\norm{H_{\psi^k}z+ D_{\phi^k}x^{k+1}}^2$ in update (\ref{z2}) can be expressed equivalently as
\[ \frac{\beta}{2}\norm{H_{\psi^k}z+ D_{\phi^k}x^{k+1}}^2= \frac{\beta}{2}\norm{H_{\psi^k}z}^2+ {\beta}( D_{\phi^k}x^{k+1})'H_{\psi^k}z+ \frac{\beta}{2}\norm{ D_{\phi^k}x^{k+1}}^2.\]
We can drop the term $ \frac{\beta}{2}\norm{ D_{\phi^k}x^{k+1}}^2$, which is constant in $z$, and write update (\ref{z2}) as
\begin{align}\label{eq:zUpdate2SplitGeneral} z^{k+1} \in \argmin_{z\in Z}   -(p^k-\beta D_{\phi^k}x^{k+1})'H_{\psi^k}z+\frac{\beta}{2}\norm{H_{\psi^k}z}^2, \end{align}
The updates (\ref{eq:xUpdate2SplitGeneral}) and (\ref{eq:zUpdate2SplitGeneral}) make the dependence on the decision variables $x$ and $z$ more explicit and therefore will be used in the convergence
analysis. We refer to (\ref{eq:xUpdate2SplitGeneral}) and (\ref{eq:zUpdate2SplitGeneral}) as the {\it primal $x$ and $z$ update} respectively, and (\ref{eq:pUpdate2SplitGeneral}) as the {\it dual
update}.

\subsection{Special Case: Distributed Multi-agent Optimization}\label{sec:specialCase}

We apply the asynchronous ADMM algorithm to the edge-based reformulation of the multi-agent optimization problem (\ref{distFormulationSync}).\footnote{For simplifying the exposition,  we assume
$n=1$ and note that the results extend to $n>1$.} Note that each constraint of this problem takes the form $x_i=x_j$ for agents $i$ and $j$ with $(i,j)\in E$. Therefore, this formulation does not
satisfy Assumption \ref{assm:matrices}.

We next introduce another reformulation of this problem, used also in  Example 4.4 of Section 3.4 in \cite{ParallelCompute}, so that each each constraint only involves one component of the decision
variable.\footnote{Note that this reformulation can be applied to any problem with a separable objective function and linear constraints to turn it into a problem of form
(\ref{Multisplit-formulation}) that satisfies Assumption \ref{assm:matrices}.} More specifically, we let $N(e)$ denote the agents which are the endpoints of edge $e$ and introduce a variable
$z=[z_{eq}]_{e=1,\ldots,M\atop q\in N(e)}$ of dimension 2$M$, one for each endpoint of each edge. Using this variable, we can write the constraint $x_i=x_j$ for each edge $e=(i,j)$ as
\[x_i=z_{ei}, \quad -x_j=z_{ej},\quad z_{ei}+z_{ej}=0.\]
The variables $z_{ei}$ can be viewed as an estimate of the component $x_j$ which is known by node $i$. The transformed problem can be written compactly as
\begin{align}\label{2splitFormulationEdge}
\min_{x_i\in X,z\in Z}\quad &\sum_{i=1}^N f_i(x_i)\\ \nonumber
s.t. \quad & A_{ei}x_{i} = z_{ei}, \quad \mbox{$e=1,\ldots,M$, $i \in \mathcal{N}(e)$,}\nonumber
\end{align}
where $Z$ is the set $\{z\in \mathbb{R}^{2M}\ |\ \sum_{q\in \mathcal{N}(e)} z_{eq}=0,\ \mbox{$e=1,\ldots,M$}\}$ and $A_{ei}$ denotes the entry in the $e^{th}$ row and $i^{th}$ column of matrix $A$,
which is either $1$ or $-1$. This formulation is in the form of problem (\ref{Multisplit-formulation}) with matrix $H = -I$, where $I$ is the identity matrix of dimension $2M\times 2M$. Matrix $D$
is of dimension $2M\times N$, where each row contains exactly one entry of $1$ or $-1$. In view of the fact that each node is incident to at least one edge, matrix $D$ has no column of all zeros.
Hence Assumption \ref{assm:matrices} is satisfied.

One natural implementation of the asynchronous algorithm is to associate with each edge an independent Poisson clock with identical rates across the edges. At iteration $k$, if the clock
corresponding to edge $(i,j)$ ticks, then $\phi^k=\{i,j\}$  and $\psi^k$ picks the rows in the constraint associated with edge $(i,j)$, i.e., the constraints $x_i=z_{ei}$ and
$-x_j=z_{ej}$.\footnote{Note that this selection is a proper partition of the constraints since the set $Z$ couples only the variables $z_{ek}$ for the endpoints of an edge $e$.}

We associate a dual variable $p_{ei}$ in $\mathbb{R}$ to each of the constraint $ A_{ei}x_{i} = z_{ei}, $ and denote the vector of dual variables by $p$. The primal $z$ update and the dual update
[Eqs.\ (\ref{z2}) and (\ref{eq:pUpdate2SplitGeneral})] for this problem are given by
\begin{align}\label{QP} z_{ei}^{k+1}, z_{ej}^{k+1} = \argmin_{z_{ei}, z_{ej}, z_{ei}+z_{ej}=0} -(p_{ei}^k)'(A_{ei}x_i^{k+1}-z_{ei})-(p_{ej}^k)'(A_{ej}x_j^{k+1}-z_{ej})\\ \nonumber+
        \frac{\beta}{2}\left(\norm{A_{ei}x_i^{k+1}-z_{ei}}^2+\norm{A_{ej}x_j^{k+1}-z_{ej}}^2\right),\end{align}
        \[ p_{eq}^{k+1} = p_{eq}^k-\beta(A_{eq}x_q^{k+1}-z_{eq}^{k+1})\qquad \hbox{for } q=i,j.\]
The primal $z$ update involves a quadratic optimization problem with linear constraints which can be solved in closed form. In particular, using first order optimality conditions, we conclude
\be
z_{ei}^{k+1} = \frac{1}{\beta}(-p_{ei}^k-v^{k+1})+A_{ei}x_i^{k+1},\qquad \qquad z_{ej}^{k+1} = \frac{1}{\beta}(-p_{ei}^k-v^{k+1})+A_{ej}x_j^{k+1},\label{eq:zej}\ee where $v^{k+1}$ is the Lagrange
multiplier associated with the constraint $z_{ei}+z_{ej}=0$ and is given by
\be \label{eq:v}v^{k+1} = \frac{1}{2}(-p_{ei}^k-p_{ej}^k)+\frac{\beta}{2}(A_{ei}x_i^{k+1}+A_{ej}x_j^{k+1}).\ee

Combining these steps yields the following asynchronous algorithm for problem (\ref{distFormulationSync}) which can be implemented in a decentralized manner by each node $i$ at each iteration $k$
having access to only his local objective function $f_i$, adjacency matrix entries $A_{ei}$, and his local variables $x_i^k$,  $z_{ei}^k$, and $p_{ei}^k$ while exchanging information with one of his
neighbors.\footnote{The asynchronous ADMM algorithm can also be applied to a node-based reformulation of problem (\ref{distopt-problem}), where we impose the local copy of each node to be equal to
the average of that of its neighbors. This leads to another asynchronous distributed algorithm with a different communication structure in which each node at each iteration broadcasts its local
variables to all his neighbors, see \cite{WOLIDS} for more details.}

\begin{framed}\noindent \textbf{II. Asynchronous Edge Based ADMM algorithm:}
\begin{itemize}
\item[A] Initialization: choose some arbitrary $x_i^0$ in $X$ and $z^0$  in $Z$, which are not necessarily all equal. Initialize $p_{ei}^0=0$ for all edges $e$ and end points $i$.
\item[B] At time step $k$, the local clock associated with edge $e =(i,j)$ ticks,
\begin{itemize}
	\item[a]  Agents $i$ and $j$ update their estimates $x_i^k$ and $x_j^k$ simultaneously as \[x_q^{k+1} = \argmin_{x_q \in X} f_q(x_q)-(p_{eq}^k)'A_{eq}x_q+\frac{\beta}{2}
\norm{A_{eq}x_q-z_{eq}^k}^2\] for $q=i,j$. The updated value of $x_i^{k+1}$ and $x_j^{k+1}$ are exchanged over the edge $e$.
 \item[b] Agents $i$ and $j$ exchange their current dual variables $p_{ei}^k$ and $p_{ej}^k$ over the edge $e$.    For $q=i,j$, agents $i$ and $j$ use the obtained values to compute the
     variable    $v^{k+1}$ as Eq.\ (\ref{eq:v}), i.e., \[v^{k+1} = \frac{1}{2}(-p_{ei}^k-p_{ej}^k)+\frac{\beta}{2}(A_{ei}x_i^{k+1}+A_{ej}x_i^{k+1}).\]
  and update their estimates $z_{ei}^k$ and $z_{ej}^k$ according to Eq.\ (\ref{eq:zej}), i.e.,
    \[z_{eq}^{k+1} = \frac{1}{\beta}(-p_{eq}^k-v^{k+1})+A_{eq}x_q^{k+1}.\]
\item [c] Agents $i$ and $j$ update the dual variables $p_{ei}^{k+1}$ and $p_{ej}^{k+1}$ as
	\[p_{eq}^{k+1}=-v^{k+1}\qquad \hbox{for }q=i,j.\]
    \item [d] All other agents keep the same variables as the previous time.\end{itemize}
\end{itemize}
\end{framed}

\section{Convergence Analysis for Asynchronous ADMM Algorithm}\label{sec:2splitconvGeneral}

In this section, we study the convergence behavior of the asynchronous ADMM algorithm under Assumptions \ref{assm:saddlePoint2Split}-\ref{assm:io}. We show that the primal iterates $\{x^k,z^k\}$ generated by (\ref{eq:xUpdate2SplitGeneral}) and
(\ref{eq:zUpdate2SplitGeneral}) converge almost surely to an optimal solution of problem (\ref{Multisplit-formulation}). Under the additional assumption that the constraint sets $X$ and $Z$ are
compact, we further show that the corresponding objective function values converge to the optimal value in expectation at rate $O(1/k)$.

We first recall the relationship between the sets $\phi^k$ and $\psi^k$ for a particular iteration $k$, which plays an important role in the analysis. Since the set of active components at time $k$,
$\phi^k$, represents all components of the decision variable that appear in the active constraints defined by the set $\psi^k$, we can write
\be\label{eq:phipsi}[Dx]_{\psi^k}=[D_{\phi^k}x]_{\psi^k}.\ee

We next consider a sequence $\{y^k,v^k,\mu^k\}$, which is formed of iterates defined by a ``full information" version of the ADMM algorithm in which all constraints (and therefore all components)
are active at each iteration.  We will show that under the Decoupled Constraints Assumption (cf.\ Assumption \ref{assm:matrices}), the iterates generated by the asynchronous algorithm
$(x^k,z^k,p^k)$ take the values of $(y^k,v^k,\mu^k)$ over the sets of active components and constraints and remain at their previous values otherwise. This association enables us to perform the
convergence analysis using the sequence $\{y^k,v^k,\mu^k\}$ and then translate the results into bounds on the objective function value improvement along the sequence $\{x^k,z^k,p^k\}$.

More specifically, at iteration $k$, we define $y^{k+1}$ by
\be\label{eq:defy} y^{k+1}\in\argmin_{y\in X} F(y)-(p^k-\beta Hz^k)'Dy+\frac{\beta}{2}\norm{Dy}^2.\ee
Due to the fact that each row of matrix $D$ has only one nonzero element [cf.\ Assumption \ref{assm:matrices}], the norm $\norm{Dy}^2$ can be decomposed as $\sum_{i=1}^N \norm{D_iy_i}^2$, where
recall that $D_i$ is the matrix that picks up the columns corresponding to component $x_i$ and is equal to zero otherwise. Thus, the preceding optimization problem can be written as a separable
optimization problem over the variables $y_i$:
\[ y^{k+1}\in \sum_{i=1}^N \argmin_{y_i\in X_i}  f_i(y_i)-(p^k-\beta Hz^k)'D_iy_i+\frac{\beta}{2}\norm{D_iy_i}^2.\]

Since $f^k(x)=\sum_{i\in \phi^k}f_i(x_i)$, and $D_{\phi^k}=\sum_{i\in \phi^k}D_i$, the minimization problem that defines the iterate $x^{k+1}$ [cf.\ Eq.\ (\ref{eq:xUpdate2SplitGeneral})] similarly
decomposes over the variables $x_i$ for $i\in \Phi^k$. Hence, the iterates $x^{k+1}$ and $y^{k+1}$ are identical over the components in set $\phi^k$, i.e., $[x^{k+1}]_{\phi^k} = [y^{k+1}]_{\phi^k}$.
Using the definition of matrix $D_{\phi^k}$, i.e., $D_{\phi^k}=\sum_{i\in\phi^k}D_i$, this implies the following relation:   \be\label{eq:dx}D_{\phi^k}x^{k+1}=D_{\phi^k}y^{k+1}.\ee The rest of the
components of the iterate $x^{k+1}$ by definition  remain at their previous value,  i.e.,  $[x^{k+1}]_{\bar \phi^k}=[x^k]_{\bar \phi^k}$.

Similarly, we define vector $v^{k+1}$ in $Z$ by
\be\label{eq:defv} v^{k+1}\in\argmin_{v\in Z} -(p^k-\beta Dy^{k+1})'Hv+\frac{\beta}{2}\norm{Hv}^2.\ee
Using the  diagonal structure of matrix $H$ [cf.\ Assumption \ref{assm:matrices}] and the fact that $\Pi$ is a proper partition of the constraint set [cf.\ Section \ref{sec:asyncAlg}], this problem
can also be decomposed in the following way:
\[ v^{k+1}\in  \argmin_{v, [v]_{\psi} \in Z_{\psi}}  \sum_{\psi\in \Pi}-(p^k-\beta Dy^{k+1})'H_\psi [v]_\psi+\frac{\beta}{2}\norm{H_\psi [v]_\psi}^2,\] where $H_\psi$ is a diagonal matrix that
contains the
$l^{th}$ diagonal element of the diagonal matrix $H$ for $l$ in set $\psi$ (and has zeros elsewhere) and set $Z_\psi$ is the projection of set $Z$ on component $[v]_\psi$. Since the diagonal matrix
$H_{\psi^k}$ has nonzero elements only on the $l^{th}$ element of the diagonal with $l\in \psi^k$,  the update of $[v]_\psi$ is independent of the other components, hence we can express the update
on the components of $v^{k+1}$ in set $\psi^k$ as
\[[v^{k+1}]_{\psi^k}\in\argmin_{v\in Z} -(p^k-\beta D_{\psi^k}x^{k+1})'H_{\psi^k}z+\frac{\beta}{2}\norm{H_{\psi^k}z}^2.\]
By the primal $z$ update [cf.\ Eq.\ (\ref{eq:zUpdate2SplitGeneral})], this shows that $[z^{k+1}]_{\psi^k}=[v^{k+1}]_{\psi^k}$.  By definition, the rest of the components of $z^{k+1}$ remain at their
previous values, i.e., $[z^{k+1}]_{\bar{\psi}^k} = [z^k]_{\bar{\psi}^k}$.

Finally, we define vector $\mu^{k+1}$ in $\mathbb{R}^W$ by
\be\label{eq:defMu} \mu^{k+1}= p^k-\beta (Dy^{k+1}+Hv^{k+1}). \ee
We relate this vector to the dual variable $p^{k+1}$ using the dual update  [cf.\ Eq.\ (\ref{eq:pUpdate2SplitGeneral})]. We also have
 \[[D_{\phi^k}x^{k+1}]_{\psi^k}=[D_{\phi^k}y^{k+1}]_{\psi^k}=[Dy^{k+1}]_{\psi^k},\]
where the first equality follows from  Eq.\ (\ref{eq:dx}) and second is derived from Eq.\ (\ref{eq:phipsi}). Moreover, since $H$ is diagonal, we have
$[H_{\psi^k}z^{k+1}]_{\psi^k}=[Hv^{k+1}]_{\psi^k}$. Thus, we obtain $[p^{k+1}]_{\psi^k}=[\mu^{k+1}]_{\psi^k}$ and $[p^{k+1}]_{\bar{\psi}^k}=[p^k]_{\bar{\psi}^k}$.

A key term in our analysis will be the {\it residual} defined at a given primal vector $(y,v)$ by
\be\label{eq:defrOthers}r=Dy+Hv.\ee
The residual term is important since its value at the primal vector $(y^{k+1},v^{k+1})$ specifies the update direction for the dual vector $\mu^{k+1}$ [cf.\ Eq.\ (\ref{eq:defMu})]. We will denote
the residual at the  primal vector $(y^{k+1},v^{k+1})$ by
 \be\label{eq:defr} r^{k+1} = Dy^{k+1}+Hv^{k+1}.\ee

\subsection{Preliminaries}

We proceed to the convergence analysis of the asynchronous algorithm. We first present some preliminary results which will be used later to establish convergence properties of asynchronous
algorithm. In particular, we provide bounds on the difference of the objective function value of the vector $y^k$ from the optimal value, the distance between $\mu^k$ and an optimal dual solution
and distance between $v^k$ and an optimal solution $z^*$. We also provide a set of sufficient conditions for a limit point of the sequence $\{x^k, z^k, p^k\}$ to be a saddle point of the Lagrangian
function. The results of this section are independent of the probability distributions of the random variables $\Phi^k$ and $\Psi^k$. Due to space constraints, the proofs of the results of in this
section are omitted. We refer the reader to \cite{WOLIDS} for the missing details.

The next lemma establishes primal feasibility (or zero residual property) of a saddle point of the Lagrangian function of problem (\ref{Multisplit-formulation}).

\begin{lemma}\label{lemma:saddlePoint2}
Let $(x^*, z^*, p^*)$ be a saddle point of the Lagrangian function defined as in Eq.\ (\ref{eq:LagGeneral}) of problem (\ref{Multisplit-formulation}). Then
\be Dx^*+H z^*=0.\label{eq:saddleFeasAll}\ee
\end{lemma}

The next theorem provides bounds on two key quantities, $F(y^{k+1})-\mu'r^{k+1}$ and $\frac{1}{2\beta}\norm{\mu^{k+1}-p^*}^2+\frac{\beta}{2}\norm{H(v^{k+1}-z^*)}^2$. {These quantities will be
related to the iterates generated by the asynchronous ADMM algorithm via a weighted norm and a weighted Lagrangian function in Section \ref{sec:convRate}. The weighted version of the quantity
$\frac{1}{2\beta}\norm{\mu^{k+1}-p^*}^2+\frac{\beta}{2}\norm{H(v^{k+1}-z^*)}^2$ is used to show almost sure convergence of the algorithm and the quantity $F(y^{k+1})-\mu'r^{k+1}$ is used in the
convergence rate analysis.}

\begin{theorem}\label{thm:deterBounds}
Let $\{x^k,z^k, p^k\}$ be the sequence generated by the asynchronous ADMM algorithm (\ref{x2})-(\ref{eq:pUpdate2SplitGeneral}).
 Let $\{y^{k}, v^{k}, \mu^{k}\}$ be the sequence defined in  Eqs.\ (\ref{eq:defy})-(\ref{eq:defMu}) and $(x^*, z^*, p^*)$ be a saddle point of the Lagrangian function of problem
 (\ref{Multisplit-formulation}). The following hold at each iteration $k$:
 \begin{align}\label{ineq:fValue} F(x^*)-F(y^{k+1})&+\mu'r^{k+1}\geq
 \frac{1}{2\beta}\left(\norm{\mu^{k+1}-\mu}^2-\norm{p^k-\mu}^2\right)\\\nonumber&+\frac{\beta}{2}\left(\norm{H(v^{k+1}-z^*)}^2-\norm{H(z^k-z^*)}^2\right)
 +\frac{\beta}{2}\norm{r^{k+1}}^2+\frac{\beta}{2}\norm{H(v^{k+1}-z^k)}^2,
\end{align}
for all $\mu$ in $\mathbb{R}^{W}$,
 and
\begin{align}\label{ineq:normValue} 0\geq &\frac{1}{2\beta}\left(\norm{\mu^{k+1}-p^*}^2-\norm{p^k-p^*}^2\right)+\frac{\beta}{2}\left(\norm{H(v^{k+1}-z^*)}^2-\norm{H(z^k-z^*)}^2\right)\\ \nonumber
&+\frac{\beta}{2}\norm{r^{k+1}}^2+\frac{\beta}{2}\norm{H(v^{k+1}-z^k)}^2.
\end{align}
\end{theorem}

The following lemma analyzes the limiting properties of the sequence $\{x^k, z^k, p^k\}$. The results will later be used in Lemma \ref{lemma:convergence}, which { provides a set of sufficient
conditions for a limit point of the sequence $\{x^k, z^k, p^k\}$ to be a saddle point.

\begin{lemma}\label{lemma:limitPoint}
Let $\{x^k,z^k, p^k\}$ be the sequence generated by the asynchronous ADMM algorithm (\ref{x2})-(\ref{eq:pUpdate2SplitGeneral}).
 Let $\{y^{k}, v^{k}, \mu^{k}\}$ be the sequence defined in  Eqs.\ (\ref{eq:defy}), (\ref{eq:defv} and )(\ref{eq:defMu}).
Consider a sample path of $\Psi^k$ and $\Phi^k$ along which the sequence $\left\{\norm{r^{k+1}}^2+\norm{H(v^{k+1}-z^k)}^2\right\}$ converges to $0$ and the sequence $\{z^k, p^k\}$ is bounded,  where $r^k$ is the residual defined as in Eq.\
(\ref{eq:defr}). Then, the sequence $\{x^k, y^k, z^k\}$  has a limit point, which is a saddle point of the Lagrangian function of problem (\ref{Multisplit-formulation}).
\end{lemma}

\subsection{Convergence and Rate of Convergence}\label{sec:convRate}

The results of the previous section did not rely on the probability distributions of random variables $\Phi^k$ and $\Psi^k$. In this section, we will introduce a weighted norm and weighted
Lagrangian function where the weights are defined in terms of the probability distributions of random variables $\Psi^k$ and $\Phi^k$ representing the active constraints and components. We will use
the weighted norm to construct a nonnegative supermartingale along the sequence $\{x^k,z^k,p^k\}$ generated by the asynchronous ADMM algorithm and use it to establish the almost sure convergence of
this sequence to a saddle point of the Lagrangian function of problem (\ref{Multisplit-formulation}). {By relating the iterates generated by the asynchronous ADMM algorithm to the variables $(y^k,
v^k, \mu^k)$ through taking expectations of the weighted Lagrangian function and using results from Theorem \ref{thm:deterBounds}, we will show that under a compactness assumption on the constraint
sets $X$ and $Z$, the asynchronous ADMM algorithm converges with rate $O(1/k)$ in expectation in terms of both objective function value and constraint violation.}

We use the notation $\alpha_i$ to denote the probability that component  $x_i$ is active at one iteration, i.e.,
\be\label{eq:defAlpha}\alpha_i= \mathbb{P}(i\in \Phi^k),\ee  and the notation $\lambda_l$ to denote the probability that constraint $l$ is active at one iteration, i.e.,
\be\label{eq:defPi}\lambda_l = \mathbb{P}(l\in \Psi^k).\ee
Note that, { since the random variables $\Phi^k$ (and $\Psi^k$) are independent and identically distributed for all $k$}, these probabilities are the same across all iterations. We define a diagonal
matrix $\Lambda$ in $\mathbb{R}^{W\times W}$ with elements $\lambda_l$ on the diagonal, i.e., \[\Lambda_{ll}=\lambda_l\qquad \hbox{for each } l\in \{1,\ldots,W\}.\] Since each constraint is assumed
to be active with strictly positive probability [cf.\ Assumption \ref{assm:io}], matrix $\Lambda$ is positive definite. We write $\bar \Lambda$ to indicate the inverse of matrix $\Lambda$. Matrix
$\bar \Lambda$ induces a {\it weighted vector norm} for $p$ in $\mathbb{R}^{W}$ as \[\norm{p}_{\bar \Lambda}^2 = p'\bar\Lambda p.\] We define a {\it weighted Lagrangian function} $\tilde
L(x,z,\mu):\mathbb{R}^{nN}\times \mathbb{R}^W\times \mathbb{R}^W\to \mathbb{R}$ as \be\label{eq:defTildeL}\tilde L(x,z,\mu)=\sum_{i=1}^N
\frac{1}{\alpha_i}f_i(x_i)-\mu'\left(\sum_{i=1}^N \frac{1}{\alpha_i}D_ix+\sum_{l=1} \frac{1}{\lambda_l}H_lz\right).\ee
We use the symbol $\mathcal{J}_k$ to denote the filtration up to and include iteration $k$, which contains information of random variables $\Phi^t$ and $\Psi^t$ for $t\leq k$. We have
$\mathcal{J}_k\subset \mathcal{J}_{k+1}$ for all $k\geq 1$.

The particular weights in $\bar \Lambda$-norm and the weighted Lagrangian function are chosen to relate the expectation of the norm
$\frac{1}{2\beta}\norm{p^{k+1}-\mu}^2_{\bar\Lambda}+\frac{\beta}{2}\norm{H(z^{k+1}-v)}^2_{\bar\Lambda}$ and function $\tilde L(x^{k+1},z^{k+1},\mu)$  to
$\frac{1}{2\beta}\norm{p^{k}-\mu}^2_{\bar\Lambda}+\frac{\beta}{2}\norm{H(z^{k}-v)}^2_{\bar\Lambda}$ and function $\tilde L(x^{k},z^{k},\mu)$, as we will show in the following lemma. This relation
will be used in Theorem \ref{thm:asyncDual} to show that the scalar sequence $\left\{\frac{1}{2\beta}\norm{p^{k}-\mu}^2_{\bar\Lambda}+\frac{\beta}{2}\norm{H(z^{k}-v)}^2_{\bar\Lambda}\right\}$ is a
nonnegative supermartingale, and establish  almost sure convergence of the asynchronous ADMM algorithm.

\begin{lemma}\label{lemma:martingale}
Let $\{x^k,z^k, p^k\}$ be the sequence generated by the asynchronous ADMM algorithm (\ref{x2})-(\ref{eq:pUpdate2SplitGeneral}).
 Let $\{y^{k}, v^{k}, \mu^{k}\}$ be the sequence defined in  Eqs.\ (\ref{eq:defy}), (\ref{eq:defv}), (\ref{eq:defMu}).
Then the following hold for each iteration $k$:
\begin{align}\label{eq:exp1}\mathbb{E}&\left(\frac{1}{2\beta}\norm{p^{k+1}-\mu}^2_{\bar\Lambda}+\frac{\beta}{2}\norm{H(z^{k+1}-v)}^2_{\bar\Lambda}\bigg|
\mathcal{J}_k\right)=\frac{1}{2\beta}\norm{\mu^{k+1}-\mu}^2\\ \nonumber &+\frac{\beta}{2}
\norm{H(v^{k+1}-v)}^2
+\frac{1}{2\beta}\norm{p^{k}-\mu}^2_{\bar\Lambda}+\frac{\beta}{2}\norm{H(z^{k}-v)}^2_{\bar\Lambda}-\frac{1}{2\beta}\norm{p^{k}-\mu}^2-\frac{\beta}{2}\norm{H(z^{k}-v)}^2,
\end{align}
for all $\mu$ in $\mathbb{R}^W$ and $v$ in $Z$, and
\begin{align}\label{eq:exp2}
\mathbb{E}&\left(\tilde L(x^{k+1},z^{k+1},\mu)\bigg| \mathcal{J}_k\right)\nonumber\\&=\left(F(y^{k+1})-\mu'(Dy^{k+1}+Hv^{k+1})\right)
+\tilde L(x^k,z^k,\mu)-\left(F(x^{k})-\mu'(Dx^{k}+Hz^{k})\right),
\end{align}
for all $\mu$ in $\mathbb{R}^W$.
\end{lemma}

\begin{proof}
By the definition of $\lambda_l$ in Eq.\ (\ref{eq:defPi}), for each $l$, the element $p^{k+1}_l$ can be either updated to $\mu^{k+1}_l$ with probability $\lambda_l$, or stay at previous value
$p_l^k$ with probability $1-\lambda_l$.  Hence, we have the following expected value
\begin{align*}\mathbb{E}\left(\frac{1}{2\beta}\norm{p^{k+1}-\mu}^2_{\bar\Lambda}\bigg| \mathcal{J}_k\right)=& \sum_{l=1}^{W}
\frac{1}{\lambda_l}\left[\lambda_l\left(\frac{1}{2\beta}\norm{\mu^{k+1}_l-\mu_l}^2\right)+(1-\lambda_l)
\left(\frac{1}{2\beta}\norm{p_l^k-\mu_l}^2\right)\right]\\\nonumber
=&\frac{1}{2\beta}\norm{\mu^{k+1}-\mu}^2 +\frac{1}{2\beta}\norm{p^{k}-\mu}^2_{\bar\Lambda}-\frac{1}{2\beta}\norm{p^{k}-\mu}^2,
\end{align*}
where the second equality follows from definition of $\norm{\cdot}_{\bar\Lambda}$, and grouping the terms.

Similarly, $z^{k+1}_l$ is either equal to $v^{k+1}_l$ with probability $\lambda_l$ or $z^{k}_l$ with probability $1-\lambda_l$. Due to the  diagonal structure of the $H$ matrix, the vector $H_lz$
has only one non-zero element equal to $[Hz]_l$ at $l^{th}$ position and zeros else where. Thus, we obtain
\begin{align*}\mathbb{E}\left(\frac{\beta}{2}\norm{H(z^{k+1}-v)}^2_{\bar\Lambda}\bigg| \mathcal{J}_k\right)&=\sum_{l=1}^{W}
\frac{1}{\lambda_l}\left[\lambda_l\left(\frac{\beta}{2}\norm{H_l(v^{k+1}-v)}\right)+(1-\lambda_l)
\left(\frac{\beta}{2}\norm{H_l(z^k-v)}^2\right)\right]
\\&=\frac{\beta}{2}
\norm{H(v^{k+1}-v)}^2
+\frac{\beta}{2}\norm{H(z^{k}-v)}^2_{\bar\Lambda}-\frac{\beta}{2}\norm{H(z^{k}-v)}^2,
\end{align*}
where we used the definition of $\norm{\cdot}_{\bar\Lambda}$ once again. By summing the above two equations and using linearity of expectation operator, we obtain Eq.\ (\ref{eq:exp1}).

Using a similar line of argument, we observe that at iteration $k$, for each $i$, $x_i^{k+1}$ has the value of $y_i^{k+1}$ with probability $\alpha_i$ and its previous value $x_l^{k}$ with
probability $1-\alpha_i$. The expectation of function $\tilde L$ therefore satisfies
\begin{align*}
\mathbb{E}&\left(\tilde L(x^{k+1},z^{k+1},\mu)\bigg| \mathcal{J}_k\right)=\sum_{i=1}^N \frac{1}{\alpha_i}\left[\alpha_i\left( f_i(y^{k+1}_i)-\mu'D_i y^{k+1}\right)+(1-\alpha_i)\left(
f_i(x^{k}_i)-\mu'D_i x^{k}\right)\right]\\ &+\sum_{l=1}^W \frac{1}{\lambda_l} \mu'\left[\lambda_l H_lv^{k+1}+(1-\lambda_l)H_lz^k\right]\\&=\left(\sum_{i=1}^N
f_i(y_i^{k+1})-\mu'(Dy^{k+1}+Hv^{k+1})\right) +\tilde L(x^k,z^k,\mu)-\left(\sum_{i=1}^N f_i(x_i^{k})-\mu'(Dx^{k}+Hz^{k})\right),
\end{align*}
where we used the fact that $D=\sum_{i=1}^N D_i$. Using the definition $F(x) = \sum_{i=1}^N f_i(x)$ [cf.\ Eq.\ (\ref{eq:defF})], this shows Eq.\ (\ref{eq:exp2}).

\end{proof}

The next lemma builds on Lemma \ref{lemma:limitPoint} and  establishes a sufficient condition for the sequence $\{x^k, z^k, p^k\}$ to converge to a saddle point of the Lagrangian. Theorem
\ref{thm:asyncDual} will then show that this sufficient condition holds with probability 1 and thus the algorithm converges almost surely.

\begin{lemma}\label{lemma:convergence}
{Let $(x^*, z^*, p^*)$ be any saddle point of the Lagrangian function of problem (\ref{Multisplit-formulation}) and $\{x^k,z^k, p^k\}$ be the sequence generated by the asynchronous ADMM algorithm
(\ref{x2})-(\ref{eq:pUpdate2SplitGeneral}). Along any sample path of $\Phi^k$ and $\Psi^k$, if the scalar sequence
$\frac{1}{2\beta}\norm{p^{k+1}-p^*}^2_{\bar\Lambda}+\frac{\beta}{2}\norm{H(z^{k+1}-z^*)}^2_{\bar\Lambda}$ is convergent and the scalar sequence
$\frac{\beta}{2}\left[\norm{r^{k+1}}^2+\norm{H(v^{k+1}-z^k)}^2\right]$ converges to $0$, then the sequence $\{x^k, z^k, p^k\}$ converges to a saddle point of the Lagrangian  function of problem
(\ref{Multisplit-formulation}).}
\end{lemma}

\begin{proof}
Since the scalar sequence $\frac{1}{2\beta}\norm{p^{k+1}-p^*}^2_{\bar\Lambda}+\frac{\beta}{2}\norm{H(z^{k+1}-z^*)}^2_{\bar\Lambda}$ converges, matrix $\bar\Lambda$ is positive definite, and matrix
$H$ is invertible [cf.\ Assumption \ref{assm:matrices}], it follows that the sequences $\{p^k\}$ and $\{z^k\}$ are bounded. Lemma \ref{lemma:limitPoint} then  implies that the sequence $\{x^k, z^k,
p^k\}$ has a limit point.

We next show that the sequence $\{x^k, z^k, p^k\}$ has a unique limit point.  Let $(\tilde x, \tilde z, \tilde p)$ be a limit point of the sequence $\{x^k, z^k, p^k\}$, i.e., the limit of sequence
$\{x^k, z^k, p^k\}$ along a subsequence $\kappa$. We first show that the components $\tilde z, \tilde p$ are uniquely defined. By Lemma \ref{lemma:limitPoint}, the point $(\tilde x, \tilde z, \tilde
p)$ is a saddle point of the Lagrangian function. Using the assumption of the lemma for $(p^*,z^*)=(\tilde p, \tilde z)$, this shows that the scalar sequence
$\left\{\frac{1}{2\beta}\norm{p^{k+1}-\tilde p}^2_{\bar\Lambda}+\frac{\beta}{2}\norm{H(z^{k+1}-\tilde z)}^2_{\bar\Lambda}\right\}$ is convergent. The limit of the sequence, therefore, is the same as
the limit along any subsequence, implying
\begin{eqnarray*}
\lim_{k\to\infty}\frac{1}{2\beta}\norm{p^{k+1}-\tilde p}^2_{\bar\Lambda}+\frac{\beta}{2}\norm{H(z^{k+1}-\tilde z)}^2_{\bar\Lambda} &=&\lim_{k\to\infty, k\in
\kappa}\frac{1}{2\beta}\norm{p^{k+1}-\tilde p}^2_{\bar\Lambda}+\frac{\beta}{2}\norm{H(z^{k+1}-\tilde z)}^2_{\bar\Lambda} \\
&=&
 \frac{1}{2\beta}\norm{\tilde p-\tilde p}^2_{\bar\Lambda}+\frac{\beta}{2}\norm{H(\tilde z-\tilde z)}^2_{\bar\Lambda} = 0,
\end{eqnarray*}
Since matrix $\tilde \Lambda$ is positive definite and matrix $H$ is invertible, this shows that  $\lim_{k\to\infty}p^k=\tilde p$ and $\lim_{k\to\infty}z^k=\tilde z$.

Next we prove that given $(\tilde z,\tilde p)$, the $x$ component of the saddle point is uniquely determined. By Lemma \ref{lemma:saddlePoint2}, we have $D\tilde x+H\tilde z=0$. Since matrix $D$ has
full column rank [cf.\ Assumption \ref{assm:matrices}], the vector $\tilde x$ is uniquely determined by $\tilde x = -(D'D)^{-1}D'H\tilde z$.
\end{proof}

The next theorem establishes almost sure convergence of the asynchronous ADMM algorithm. Our analysis uses  results related to supermartingales (interested readers are referred to \cite{Grimmett}
and \cite{williamsMartingale} for a comprehensive treatment of the subject).

\begin{theorem}\label{thm:asyncDual}
Let $\{x^k,z^k, p^k\}$ be the sequence generated by the asynchronous ADMM algorithm  (\ref{x2})-(\ref{eq:pUpdate2SplitGeneral}). The sequence $\{x^k, z^k, p^k\}$ converges almost surely to a saddle
point of the Lagrangian function of problem (\ref{Multisplit-formulation}).
\end{theorem}

\begin{proof}
We will show that the conditions of  Lemma \ref{lemma:convergence} are satisfied almost surely. We will first focus on the scalar sequence
$\frac{1}{2\beta}\norm{p^{k+1}-\mu}^2_{\bar\Lambda}+\frac{\beta}{2}\norm{H(z^{k+1}-v)}^2_{\bar\Lambda}$ and show that it is a nonnegative supermartingale. By martingale convergence theorem, this
shows that it converges almost surely. We next establish that the scalar sequence $\frac{\beta}{2}\norm{r^{k+1}-H(v^{k+1}-z^k)}^2$ converges to 0 almost surely by an argument similar to the one used
to establish Borel-Cantelli lemma. These two results imply that the set of events where $\frac{1}{2\beta}\norm{p^{k+1}-p^*}^2_{\bar\Lambda}+\frac{\beta}{2}\norm{H(z^{k+1}-z^*)}^2_{\bar\Lambda}$ is
convergent and $\frac{\beta}{2}\left[\norm{r^{k+1}}^2+\norm{H(v^{k+1}-z^k)}^2\right]$ converges to $0$ has probability $1$. Hence, by Lemma \ref{lemma:convergence}, we have the sequence $\{x^k, z^k,
p^k\}$ converges to a saddle point of the Lagrangian function almost surely.

We first show that the scalar sequence $\frac{1}{2\beta}\norm{p^{k+1}-\mu}^2_{\bar\Lambda}+\frac{\beta}{2}\norm{H(z^{k+1}-v)}^2_{\bar\Lambda}$ is a nonnegative supermartingale. Since it is a
summation of two norms, it immediately follows that it is nonnegative. To see it is a supermartingale, we let vectors $y^{k+1},v^{k+1},\mu^{k+1}$ and $r^{k+1}$ be those defined in Eqs.\
(\ref{eq:defy}), (\ref{eq:defv}), (\ref{eq:defMu}) and (\ref{eq:defr}). Recall that the symbol $\mathcal{J}_k$ denotes the filtration up to and including iteration $k$. From Lemma \ref{lemma:martingale}, we have
\begin{align*}&\mathbb{E}\left(\frac{1}{2\beta}\norm{p^{k+1}-\mu}^2_{\bar\Lambda}+\frac{\beta}{2}\norm{H(z^{k+1}-v)}^2_{\bar\Lambda}\bigg| \mathcal{J}_k\right)\\ \nonumber&\qquad
=\frac{1}{2\beta}\norm{\mu^{k+1}-\mu}^2+\frac{\beta}{2}
\norm{H(v^{k+1}-v)}^2
+\frac{1}{2\beta}\norm{p^{k}-\mu}^2_{\bar\Lambda}+\frac{\beta}{2}\norm{H(z^{k}-v)}^2_{\bar\Lambda}\\ \nonumber& \qquad \quad -\frac{1}{2\beta}\norm{p^{k}-\mu}^2-\frac{\beta}{2}\norm{H(z^{k}-v)}^2
\end{align*}
Substituting  $\mu=p^*$ and $v=z^*$ in the above expectation calculation and combining with the following  inequality from Theorem \ref{thm:deterBounds},
\begin{align*} 0\geq &\frac{1}{2\beta}\left(\norm{\mu^{k+1}-p^*}^2-\norm{p^k-p^*}^2\right)+\frac{\beta}{2}\left(\norm{H(v^{k+1}-z^*)}^2-\norm{H(z^k-z^*)}^2\right)\\
&+\frac{\beta}{2}\norm{r^{k+1}}^2+\frac{\beta}{2}\norm{H(v^{k+1}-z^k)}^2,
\end{align*}
we obtain
\begin{align*}&\mathbb{E}\left(\frac{1}{2\beta}\norm{p^{k+1}-p^*}^2_{\bar\Lambda}+\frac{\beta}{2}\norm{H(z^{k+1}-z^*)}^2_{\bar\Lambda}\bigg| \mathcal{J}_k\right)
\\&\leq\frac{1}{2\beta}\norm{p^{k}-p^*}^2_{\bar\Lambda}+\frac{\beta}{2}\norm{H(z^{k}-z^*)}^2_{\bar\Lambda}
-\frac{\beta}{2}\norm{r^{k+1}}^2-\frac{\beta}{2}\norm{H(v^{k+1}-z^k)}^2.
\end{align*}
Hence, the sequence $\frac{1}{2\beta}\norm{p^{k+1}-p^*}^2_{\bar\Lambda}+\frac{\beta}{2}\norm{H(z^{k+1}-z^*)}^2_{\bar\Lambda}$ is a nonnegative supermartingale in $k$ and by martingale convergence
theorem, it  converges almost surely.

We next establish that the scalar sequence $\left\{\frac{\beta}{2}\norm{r^{k+1}}^2+\frac{\beta}{2}\norm{H(v^{k+1}-z^k)}^2\right\}$ converges to 0 almost surely. Rearranging the terms {in the
previous inequality} and taking iterated expectation with respect to the filtration $\mathcal{J}_k$, we obtain for all $T$
\begin{align}\label{ineq:expInner}\sum_{k=1}^T\mathbb{E}\left(\frac{\beta}{2}\norm{r^{k+1}}^2+\frac{\beta}{2}\norm{H(v^{k+1}-z^k)}^2\right)
\leq&\frac{1}{2\beta}\norm{p^{0}-p^*}^2_{\bar\Lambda}+\frac{\beta}{2}\norm{H(z^{0}-z^*)}^2_{\bar\Lambda}\\\nonumber
&-\mathbb{E}\left(\frac{1}{2\beta}\norm{p^{T+1}-p^*}^2_{\bar\Lambda}+\frac{\beta}{2}\norm{H(z^{T+1}-z^*)}^2_{\bar\Lambda}\right)\\\nonumber
\leq&\frac{1}{2\beta}\norm{p^{0}-p^*}^2_{\bar\Lambda}+\frac{\beta}{2}\norm{H(z^{0}-z^*)}^2_{\bar\Lambda},
\end{align}
where the last inequality follows from relaxing the upper bound by dropping the non-positive expected value term. Thus, the sequence
$\left\{\mathbb{E}\left(\frac{\beta}{2}\left[\norm{r^{k+1}}^2+\norm{H(v^{k+1}-z^k)}^2\right]\right)\right\}$ is summable implying
\be\label{eq:summable}\lim_{k\to \infty} \sum_{t=k}^\infty \mathbb{E}\left(\frac{\beta}{2}\left[\norm{r^{k+1}}^2+\norm{H(v^{k+1}-z^k)}^2\right]\right) = 0\ee By Markov inequality, we have
\[\mathbb{P}\left(\frac{\beta}{2}\left[\norm{r^{k+1}}^2+\norm{H(v^{k+1}-z^k)}^2\right] \geq \epsilon\right)\leq
\frac{1}{\epsilon}\mathbb{E}\left(\frac{\beta}{2}\left[\norm{r^{k+1}}^2+\norm{H(v^{k+1}-z^k)}^2\right]\right),\]
for any scalar $\epsilon>0$ for all iterations $t$. Therefore, we have
\begin{align*}\lim_{k\to\infty} \mathbb{P}&\left(\sup_{t\geq k}\frac{\beta}{2}\left[\norm{r^{k+1}}^2+\norm{H(v^{k+1}-z^k)}^2\right] \geq
\epsilon\right)\\&=\lim_{k\to\infty}\mathbb{P}\left(\bigcup_{t=k}^\infty \frac{\beta}{2}\left[\norm{r^{k+1}}^2+\norm{H(v^{k+1}-z^k)}^2\right]\geq \epsilon\right)\\
&\leq \lim_{k\to \infty} \sum_{t=k}^\infty \mathbb{P}\left(\frac{\beta}{2}\left[\norm{r^{k+1}}^2+\norm{H(v^{k+1}-z^k)}^2\right] \geq \epsilon\right)\\ &\leq\lim_{k\to \infty}
\frac{1}{\epsilon}\sum_{t=k}^\infty \mathbb{E}\left(\frac{\beta}{2}\left[\norm{r^{k+1}}^2+\norm{H(v^{k+1}-z^k)}^2\right]\right)=0,\end{align*} where the first inequality follows from union
bound on probability, the second inequality follows from the preceding relation, and the last equality follows from Eq.\ (\ref{eq:summable}). This proves that the sequence
$\frac{\beta}{2}\left[\norm{r^{k+1}}^2+\norm{H(v^{k+1}-z^k)}^2\right]$ converges to $0$ almost surely.
\end{proof}

We next analyze convergence rate of the asynchronous ADMM algorithm. The rate analysis is done with respect to the time ergodic averages defined as $\bar x(T)$ in $\mathbb{R}^{nN}$, the time average
of $x^k$ up to and including iteration $T$, i.e.,
\be\label{eq:defbarX}\bar x_i(T)=\frac{\sum_{1=1}^T x_i^{k}}{T},\ee for all $i=1,\ldots, N$,\footnote{Here the notation $\bar x_i(T)$ denotes the vector of length $n$ corresponding to agent $i$.}
 and $\bar z(k)$ in $\mathbb{R}^{W}$ as
\be\label{eq:defbarZ}\bar z_l(T)=\frac{\sum_{k=1}^T z_{l}^{k}}{T},\ee for all $l=1, \ldots, W$.

We next introduce some scalars $Q(\mu)$, $\bar Q$, $\bar \theta$ and $\tilde L^0$, all of which will be used to provide an upper bound on the constant term that appears in the rate analysis. Scalar
$Q(\mu)$ is defined by
\be\label{eq:defQ}Q(\mu)=\max_{x\in X, z\in Z}-\tilde L(x, z,\mu),\ee
which implies $Q(\mu)\geq -\tilde L(x^{k+1}, z^{k+1},\mu)$ for any realization of $\Psi^k$ and $\Phi^k$.
For the rest of the section, we adopt the following assumption, which will be used to guarantee that scalar $Q(\mu)$ is well defined and finite:
\begin{assumption}
\label{assm:compact}
The sets $X$ and $Z$ are both compact.
\end{assumption}

Since the weighted Lagrangian function $\tilde L$ is continuous in $x$ and $z$ [cf.\ Eq.\ (\ref{eq:defTildeL})], and all iterates $(x^k, z^k)$ are in the compact set  $X\times  Z$, by Weierstrass
theorem the maximization in the preceding equality is attained and finite. 

Since function $\tilde L$ is linear in $\mu$, the function $Q(\mu)$ is the maximum of linear functions and  is  thus convex and continuous  in $\mu$. We define scalar $\bar Q$ as
\be\label{eq:defBarQ} \bar Q=\max_{\mu=p^*-\alpha,\norm{\alpha}\leq 1} Q(\mu).\ee
The reason that such scalar $\bar Q<\infty$ exists is once again by Weierstrass theorem (maximization over a compact set).

We define vector $\bar\theta$ in $\mathbb{R}^W$ as \be\label{eq:defbartheta}\bar\theta = p^* - \argmax_{\norm{u}\leq 1}\norm{p^0-(p^*-u)}_{\bar\Lambda}^2,\ee
 such maximizer exists due to Weierstrass theorem and the fact that the set $\norm{u}\leq 1$ is compact and the function $\norm{p^0-(p^*-u)}_{\bar\Lambda}^2$ is continuous.  Scalar $\tilde L^0$ is
 defined by
\be\label{eq:deftildeL0} \tilde L^0=\max_{\theta=p^*-\alpha, \norm{\alpha}\leq 1} \tilde L(x^0,z^0,\theta).\ee This scalar is well defined because the constraint set is compact and the function
$\tilde L$ is continuous in $\theta$.

\begin{theorem}\label{thm:rate}
Let $\{x^k,z^k, p^k\}$ be the sequence generated by the asynchronous ADMM algorithm \ref{x2})-(\ref{eq:pUpdate2SplitGeneral}) and  $(x^*, z^*, p^*)$ be a saddle point of the Lagrangian function of
problem  (\ref{Multisplit-formulation}). Let the vectors $\bar x(T)$, $\bar z(T)$ be defined as in Eqs.\ (\ref{eq:defbarX}) and (\ref{eq:defbarZ}), the scalars $\bar Q$, $\bar \theta$  and $\tilde
L^0$ be defined as in Eqs. (\ref{eq:defBarQ}), \ (\ref{eq:defbartheta}) and (\ref{eq:deftildeL0}) and the function $\tilde L$ be defined as in Eq.\ (\ref{eq:defTildeL}). Then the following relations
hold:
\begin{align}\label{ineq:rateFeasiblity}
\norm{\mathbb{E}(D \bar x(T)+H\bar z(T)) } \leq \frac{1}{T}\left[\bar Q
 +\tilde L^0
+\frac{1}{2\beta}\norm{p^{0}-\bar\theta}^2_{\bar\Lambda}+\frac{\beta}{2}\norm{H(z^{0}-z^*)}^2_{\bar\Lambda}\right],\end{align} and
\begin{align}\label{ineq:ratePrimal}\norm{\mathbb{E}(F(\bar x(T)))-F(x^*)} \leq&  \frac{1}{T}\left[\bar Q
 +\tilde L^0
+\frac{1}{2\beta}\norm{p^{0}-p^*}^2_{\bar\Lambda}+\frac{\beta}{2}\norm{H(z^{0}-z^*)}^2_{\bar\Lambda}\right]\\\nonumber&+ \frac{\norm{p^*}_\infty}{T}\left[Q(p^*)
 +\tilde L(x^0,z^0,p^*)
+\frac{1}{2\beta}\norm{p^{0}-\bar\theta}^2_{\bar\Lambda}+\frac{\beta}{2}\norm{H(z^{0}-z^*)}^2_{\bar\Lambda}\right].\end{align}
\end{theorem}

\begin{proof}
The proof of the theorem relies on Lemma \ref{lemma:martingale} and Theorem \ref{thm:deterBounds}. We combine these results with law of iterated expectation, telescoping cancellation and convexity
of the function $F$ to establish the bound
 \begin{align}\label{ineq:rate2}\mathbb{E}\left[F(\bar x(T))\right.&\left.-\mu'(D \bar x(T)+H\bar z(T))\right]-F(x^*) \\ \nonumber&\leq \frac{1}{T}\left[Q(\mu)
 +\tilde L(x^0,z^0,\mu)
+\frac{1}{2\beta}\norm{p^{0}-\mu}^2_{\bar\Lambda}+\frac{\beta}{2}\norm{H(z^{0}-z^*)}^2_{\bar\Lambda}\right],
\end{align}
for all $\mu$ in $\mathbb{R}^W$. Then by using different choices of the vector $\mu$, we obtain the desired results.

We will first prove Eq.\ (\ref{ineq:rate2}). Recall Eq.\ (\ref{eq:exp2}):
\begin{align*}
\mathbb{E}&\left(\tilde L(x^{k+1},z^{k+1},\mu)\bigg| \mathcal{J}_k\right)\nonumber\\&=\left(F(y^{k+1})-\mu'(Dy^{k+1}+Hv^{k+1})\right)
+\tilde L(x^k,z^k,\mu)-\left(F(x^{k})-\mu'(Dx^{k}+Hz^{k})\right),
\end{align*}
  We rearrange Eq.\ (\ref{ineq:fValue}) from Theorem \ref{thm:deterBounds}, and obtain
\begin{align*} F(y^{k+1})&-\mu'r^{k+1}\leq F(x^*)-
\frac{1}{2\beta}\left(\norm{\mu^{k+1}-\mu}^2-\norm{p^k-\mu}^2\right)\\\nonumber&-\frac{\beta}{2}\left(\norm{H(v^{k+1}-z^*)}^2-\norm{H(z^k-z^*)}^2\right)
-\frac{\beta}{2}\norm{r^{k+1}}^2-\frac{\beta}{2}\norm{H(v^{k+1}-z^k)}^2.
\end{align*}
Since $r^{k+1}=Dy^{k+1}+Hv^{k+1}$, we can apply this bound on the first term on the right-hand side of the preceding relation which implies
\begin{align*}
\mathbb{E}&\left(\tilde L(x^{k+1},z^{k+1},\mu)\bigg| \mathcal{J}_k\right)\leq  F(x^*)-
\frac{1}{2\beta}\left(\norm{\mu^{k+1}-\mu}^2-\norm{p^k-\mu}^2\right)\\\nonumber&-\frac{\beta}{2}\left(\norm{H(v^{k+1}-z^*)}^2-\norm{H(z^k-z^*)}^2\right)
-\frac{\beta}{2}\norm{r^{k+1}}^2-\frac{\beta}{2}\norm{H(v^{k+1}-z^k)}^2\\ &+\tilde L(x^k,z^k,\mu)-\left(F(x^{k})-\mu'(Dx^{k}+Hz^{k})\right),
\end{align*}
Combining the above inequality with Eq.\ (\ref{eq:exp1}) and using the linearity of expectation, we have
\begin{align*}\mathbb{E}&\left(\tilde L(x^{k+1},z^{k+1},\mu)+\frac{1}{2\beta}\norm{p^{k+1}-\mu}^2_{\bar\Lambda}+\frac{\beta}{2}\norm{H(z^{k+1}-z^*)}^2_{\bar\Lambda}\bigg| \mathcal{J}_k\right)\\
\leq&  F(x^*) -\frac{\beta}{2}\norm{r^{k+1}}^2-\frac{\beta}{2}\norm{H(v^{k+1}-z^k)}^2-\left(F(x^{k})-\mu'(Dx^{k}+Hz^{k})\right)\\ &+\tilde L(x^k,z^k,\mu)
+\frac{1}{2\beta}\norm{p^{k}-\mu}^2_{\bar\Lambda}+\frac{\beta}{2}\norm{H(z^{k}-z^*)}^2_{\bar\Lambda}\\
\leq &F(x^*)-\left(F(x^{k})-\mu'(Dx^{k}+Hz^{k})\right)+\tilde L(x^k,z^k,\mu)
+\frac{1}{2\beta}\norm{p^{k}-\mu}^2_{\bar\Lambda}+\frac{\beta}{2}\norm{H(z^{k}-z^*)}^2_{\bar\Lambda},
\end{align*}
where the last inequality follows from relaxing the upper bound by dropping the non-positive term $-\frac{\beta}{2}\norm{r^{k+1}}^2-\frac{\beta}{2}\norm{H(v^{k+1}-z^k)}^2$.

This relation holds for $k=1,\ldots, T$ and by the law of iterated expectation, the telescoping sum after term cancellation satisfies
\begin{align}\label{ineq:rateExp}\mathbb{E}&\left(\tilde L(x^{T+1},z^{T+1},\mu)+\frac{1}{2\beta}\norm{p^{T+1}-\mu}^2_{\bar\Lambda}+\frac{\beta}{2}\norm{H(z^{T+1}-z^*)}^2_{\bar\Lambda}\right)
\leq TF(x^*)\\\nonumber
&-\mathbb{E}\left[\sum_{k=1}^T\left(F(x^{k})-\mu'(Dx^{k}+Hz^{k})\right)\right]+\tilde L(x^0,z^0,\mu)
+\frac{1}{2\beta}\norm{p^{0}-\mu}^2_{\bar\Lambda}+\frac{\beta}{2}\norm{H(z^{0}-z^*)}^2_{\bar\Lambda}.
\end{align}
By convexity of the functions $f_i$, we have
\[\sum_{k=1}^T F(x^k) =  \sum_{k=1}^T \sum_{i=1}^N f_i(x_i^k)\geq T \sum_{i=1}^N f_i(\bar x_i(T))=TF(\bar x(T)).\]
The same results hold after taking expectation on both sides. By linearity of matrix-vector multiplication, we have $\sum_{k=1}^T Dx^k=TD\bar x(T),\ \sum_{k=1}^T Hz^k=TH\bar z(T).$ Relation
(\ref{ineq:rateExp}) therefore implies that
\begin{align*}T\mathbb{E}&\left[F(\bar x(T))-\mu'(D \bar x(T)+H\bar z(T))\right]-TF(x^*)\\ \leq&
\mathbb{E}\left[\sum_{k=1}^T\left(F(x^{k})-\mu'(Dx^{k}+Hz^{k})\right)\right]
-TF(x^*)\\ \leq&-\mathbb{E}\left(\tilde L(x^{T+1},z^{T+1},\mu)+\frac{1}{2\beta}\norm{p^{T+1}-\mu}^2_{\bar\Lambda}+\frac{\beta}{2}\norm{H(z^{T+1}-z^*)}^2_{\bar\Lambda}\right)
\\& +\tilde L(x^0,z^0,\mu)
+\frac{1}{2\beta}\norm{p^{0}-\mu}^2_{\bar\Lambda}+\frac{\beta}{2}\norm{H(z^{0}-z^*)}^2_{\bar\Lambda}.
\end{align*}

Using the definition of scalar $Q(\mu)$ [cf.\ Eq.\ (\ref{eq:defQ})] and by dropping the non-positive norm terms from the above upper bound, we obtain
\begin{align*}T\mathbb{E}\left[F(\bar x(T))\right.&\left.-\mu'(D \bar x(T)+H\bar z(T))\right]-TF(x^*)\\& \leq Q(\mu)
 +\tilde L(x^0,z^0,\mu)
+\frac{1}{2\beta}\norm{p^{0}-\mu}^2_{\bar\Lambda}+\frac{\beta}{2}\norm{H(z^{0}-z^*)}^2_{\bar\Lambda}.
\end{align*}
We now divide both sides of the preceding inequality by $T$ and obtain Eq.\ (\ref{ineq:rate2}).

We now use Eq.\ (\ref{ineq:rate2}) to first show that ${\norm{\mathbb{E}(D\bar x(T)+H\bar z(T))}}$ converges to 0 with rate $1/T$.  For each iteration $T$, we define a vector $\theta(T)$ as
$\theta(T) = p^*-\frac{\mathbb{E}(D\bar x(T)+H\bar z(T))}{\norm{\mathbb{E}(D\bar x(T)+H\bar z(T))}}$. By substituting $\mu=\theta(T)$ in Eq.\ (\ref{ineq:rate2}), we obtain for each $T$,
\begin{align*} \mathbb{E}&\left[F(\bar x(T))-(\theta(T))'(D \bar x(T)+H\bar z(T))\right]-F(x^*) \\&\leq \frac{1}{T}\left[Q(\theta(T))
 +\tilde L(x^0,z^0, \theta(T))
+\frac{1}{2\beta}\norm{p^{0}-\theta(T)}^2_{\bar\Lambda}+\frac{\beta}{2}\norm{H(z^{0}-z^*)}^2_{\bar\Lambda}\right],\end{align*}

Since the vectors $\frac{\mathbb{E}(D\bar x(T)+H\bar z(T))}{\norm{\mathbb{E}(D\bar x(T)+H\bar z(T))}}$ all have norm $1$ and hence $\theta(T)$ is bounded within the unit sphere, by using the
definition of $\bar \theta$, we have $\norm{p^0-\theta(T)}_{\bar\Lambda}^2\leq \norm{p^0-\bar\theta}_{\bar\Lambda}^2$. Eqs.\ (\ref{eq:defBarQ}) and (\ref{eq:deftildeL0}) implies $Q(\theta(T))\leq
\bar Q$ and $\tilde L(x^0, z^0, \theta(T))\leq \tilde L^0$ for all $T$. Thus the above inequality suggests that the following holds true for all $T$,
\begin{align*} \mathbb{E}(F(\bar x(T))-(\theta(T))'\mathbb{E}(D \bar x(T)+H\bar z(T))-F(x^*) \leq \frac{1}{T}\left[\bar Q
 +\tilde L^0
+\frac{1}{2\beta}\norm{p^{0}-\bar\theta}^2_{\bar\Lambda}+\frac{\beta}{2}\norm{H(z^{0}-z^*)}^2_{\bar\Lambda}\right].\end{align*}

From the definition of $\theta(T)$, we have $(\theta(T))'\mathbb{E}(D\bar x(T)+H\bar z(T)) = (p^*)'\mathbb{E}(D\bar x(T)+H\bar z(T)) - \norm{\mathbb{E}(D\bar x(T)+H\bar z(T))}$, and thus
\begin{align*}\mathbb{E}(F(\bar x(T))-(\theta(T))'\mathbb{E}(D \bar x(T)+H\bar z(T))-F(x^*) =&\mathbb{E}(F(\bar x(T)))-(p^*)'\mathbb{E}\left[(D \bar x(T)+H\bar
z(T))\right]\\&-F(x^*)+\norm{\mathbb{E}(D \bar x(T)+H\bar z(T)) }.\end{align*}

Since the point $(x^*, z^*, p^*)$ is a saddle point of the Lagrangian function, using Lemma \ref{lemma:saddlePoint2}, we have
\be\label{ineq:lag}0\le \mathbb{E}L((\bar x(T)), \bar z(T), p^*) - L(x^*, z^*, p^*) = \mathbb{E}(F(\bar x(T))) - F(x^*)-(p^*)'\mathbb{E}\left[(D \bar x(T)+H\bar z(T))\right].\ee
The preceding three relations imply that
\begin{align*} \norm{\mathbb{E}(D \bar x(T)+H\bar z(T)) } \leq \frac{1}{T}\left[\bar Q
 +\tilde L^0
+\frac{1}{2\beta}\norm{p^{0}-\bar\theta}^2_{\bar\Lambda}+\frac{\beta}{2}\norm{H(z^{0}-z^*)}^2_{\bar\Lambda}\right],\end{align*} which shows the first desired inequality.

To prove Eq.\ (\ref{ineq:ratePrimal}), we let $\mu=p^*$ in Eq.\ (\ref{ineq:rate2}) and obtain
\begin{align*} \mathbb{E}(F(\bar x(T))-&(p^*)'\mathbb{E}(D \bar x(T)+H\bar z(T))-F(x^*) \\&\leq \frac{1}{T}\left[Q(p^*)
 +\tilde L(x^0,z^0, p^*)
+\frac{1}{2\beta}\norm{p^{0}-p^*}^2_{\bar\Lambda}+\frac{\beta}{2}\norm{H(z^{0}-z^*)}^2_{\bar\Lambda}\right].
\end{align*}
This inequality together with Eq.\ (\ref{ineq:lag}) imply
\begin{align*}&\norm{\mathbb{E}(F(\bar x(T)))-(p^*)'\mathbb{E}(D \bar x(T)+H\bar z(T))-F(x^*)} \\ &\leq  \frac{1}{T}\left[Q(p^*)
 +\tilde L(x^0,z^0, p^*)
+\frac{1}{2\beta}\norm{p^{0}-p^*}^2_{\bar\Lambda}+\frac{\beta}{2}\norm{H(z^{0}-z^*)}^2_{\bar\Lambda}\right].\end{align*} By triangle inequality, we obtain
\begin{align}\label{ineq:fNorm}\norm{\mathbb{E}(F(\bar x(T)))-F(x^*)} \leq&  \frac{1}{T}\left[Q(p^*)
 +\tilde L(x^0,z^0,p^*)
+\frac{1}{2\beta}\norm{p^{0}-p^*}^2_{\bar\Lambda}+\frac{\beta}{2}\norm{H(z^{0}-z^*)}^2_{\bar\Lambda}\right]\\\nonumber&+\norm{\mathbb{E}((p^*)'(D \bar x(T)+H\bar z(T)))},\end{align} Using definition
of Euclidean and $l_\infty$ norms,\footnote{We use the standard notation that $\norm{x}_\infty = \max_i |x_i|$.} the last term $\norm{\mathbb{E}((p^*)'(D \bar x(T)+H\bar z(T)))}$ satisfies
\begin{align*}\norm{\mathbb{E}((p^*)'(D \bar x(T)+H\bar z(T)))} &= \sqrt{\sum_{l=1}^W (p_l^*)^2[\mathbb{E}(D\bar x(T)+H\bar z(T))]_l^2} \\&\leq \sqrt{\sum_{l=1}^W
\norm{p^*}_\infty^2[\mathbb{E}(D\bar x(T)+H\bar z(T))]_l^2} =\norm{p^*}_\infty\norm{\mathbb{E}(D\bar x(T)+H\bar z(T))}.\end{align*}
The above inequality combined with Eq.\ (\ref{ineq:rateFeasiblity}) yields
\[\norm{\mathbb{E}((p^*)'(D \bar x(T)+H\bar z(T)))} \leq \frac{\norm{p^*}_\infty}{T}\left[Q(p^*)
 +\tilde L(x^0,z^0,p^*)
+\frac{1}{2\beta}\norm{p^{0}-\bar\theta}^2_{\bar\Lambda}+\frac{\beta}{2}\norm{H(z^{0}-z^*)}^2_{\bar\Lambda}\right].\] Hence, Eq.\ (\ref{ineq:fNorm}) implies
\begin{align*}\norm{\mathbb{E}(F(\bar x(T)))-F(x^*)} \leq&  \frac{1}{T}\left[\bar Q
 +\tilde L^0
+\frac{1}{2\beta}\norm{p^{0}-p^*}^2_{\bar\Lambda}+\frac{\beta}{2}\norm{H(z^{0}-z^*)}^2_{\bar\Lambda}\right]\\\nonumber&+ \frac{\norm{p^*}_\infty}{T}\left[Q(p^*)
 +\tilde L(x^0,z^0,p^*)
+\frac{1}{2\beta}\norm{p^{0}-\bar\theta}^2_{\bar\Lambda}+\frac{\beta}{2}\norm{H(z^{0}-z^*)}^2_{\bar\Lambda}\right].\end{align*} Thus we have established the desired relation
(\ref{ineq:ratePrimal}).
\end{proof}

We remark that by Jensen's inequality and convexity of the function $F$, we have
 $F(\mathbb{E}(\bar x(T)))\leq \mathbb{E}(F(\bar x(T))),$ and the preceding results also holds true when we replace $\mathbb{E}(F(\bar x(T)))$ by $ F(\mathbb{E}(\bar x(T)))$.

\section{Conclusions}\label{sec:con}
We developed a fully asynchronous ADMM based algorithm for a convex optimization problem with separable objective function and linear constraints. This problem is motivated by distributed
multi-agent optimization problems where a (static) network of agents each with access to a privately known local objective function seek to optimize the sum of these functions using computations
based on local information and communication with neighbors. We show that this algorithm converges almost surely to an optimal solution. Moreover, the rate of convergence of the objective function
values and feasibility violation is given by $O(1/k)$. Future work includes investigating network effects (e.g., effects of communication noise, quantization) and time-varying network topology on
the performance of the algorithm.

\bibliographystyle{plain}
\bibliography{citationSplitting}

\end{document}